\documentclass[a4paper,10pt,leqno, english]{amsart}
\title[Twisted K-theory]
{Twisted  Equivariant  $K$-Theory  and  $K$-Homology  of $\sldrei$}
\author{No\'{e} B\'{a}rcenas }
        \address{
Hausdorff  Center  for  Mathematics \\ Endenicherallee, 60 53115 Bonn, Germany  \\ and \\Centro  de  Ciencias Matem\'aticas. \\ Universidad  Nacional Aut\'onoma  de  M\'exico \\ Morelia, Michoac\'an.  M\'exico 58089}
         \email{barcenas@matmor.unam.mx}
         \urladdr{http://www.matmor.unam.mx/~barcenas}

\author{Mario Vel\'asquez}
\address{Centro  de  Ciencias Matem\'aticas. \\ Universidad  Nacional Aut\'onoma  de  M\'exico \\ Morelia, Michoac\'an.  M\'exico 58089}
 \email{mavelasquezm@gmail.com}
         \date{\today}

\usepackage{hyperref}
\usepackage{verbatim}
\usepackage{pdfsync}
\usepackage{calc}
\usepackage{enumerate,amssymb}
\usepackage[arrow,curve,matrix,tips,2cell]{xy}
  \SelectTips{eu}{10} \UseTips
  \UseAllTwocells
\usepackage{graphicx}
\usepackage{pb-diagram}

\DeclareMathAlphabet\EuR{U}{eur}{m}{n}
\SetMathAlphabet\EuR{bold}{U}{eur}{b}{n}

\makeindex             % used for the subject index
\usepackage{pstricks}
\usepackage{pst-3dplot}

\theoremstyle{plain}
\newtheorem{theorem}{Theorem}[section]

\newtheorem{proposition}[theorem]{Proposition}
\newtheorem{corollary}[theorem]{Corollary}

\theoremstyle{definition}
\newtheorem{definition}[theorem]{Definition}

\newtheorem{condition}[theorem]{Conditions}
\newtheorem{remark}[theorem]{Remark}

\newtheorem*{theoremn}{Theorem}
\newtheorem*{corollaryn}{Corollary}
{\catcode`@=11\global\let\c@equation=\c@theorem}

\newcommand{\comsquare}[8]                   % Produces a commutative square
{\begin{CD}
#1 @>#2>> #3\\
@V{#4}VV @V{#5}VV\\
#6 @>#7>> #8
\end{CD}
}

\newcommand{\xycomsquare}[8]                   % kommutatives Quadrat (xy-Version)
{\xymatrix
{#1 \ar[r]^{#2} \ar[d]^{#4} &
#3 \ar[d]^{#5}  \\
#6\ar[r]^{#7} &
#8
}
}

\newcommand{\calfin}{\mathcal{FIN}}

\newcommand{\UU}{\mathcal{U}}
\newcommand{\HH}{\mathcal{H}}
\newcommand{\KK}{\mathcal{K}}

\newcommand{ \calr}{\mathcal{R}}
\newcommand{\IC}{{\mathbb C}}

\newcommand{\IN}{{\mathbb N}}

\newcommand{\IZ}{{\mathbb Z}}

\newcommand{\curs}{\EuR}

\newcommand{\CHAINCOMPLEXES}{\curs{CHCOM}}

\newcommand{\MODULES}{\curs{MODULES}}

\newcommand{\coker}{\operatorname{coker}}
\newcommand{\colim}{\operatorname{colim}}

\newcommand{\Ext}{\operatorname{Ext}}

\newcommand{\im}{\operatorname{im}}

\newcommand{\sldrei}{Sl_{3}{\mathbb{Z}}}

\newcommand{\so}[2]{u_{#1}^ {#2}}

\newcommand{\eub}[1]{\underline{E}#1}              %Eunderbar G = classifying space for proper G-actions
\newcommand{\higherlim}[3]{{\setbox1=\hbox{\rm lim}
        \setbox2=\hbox to \wd1{\leftarrowfill} \ht2=0pt \dp2=-1pt
        \mathop{\vtop{\baselineskip=5pt\box1\box2}}
        _{#1}}^{#2}#3}

\newcommand{\OO}{\mathcal{O}_{G}}
\newcommand{\version}[1]                       %marks the date of last editing and compilation
{\begin{center} last edited on #1\\
last compiled on \today \\
name of texfile: \jobname
\end{center}
}

\newcounter{commentcounter}

\begin{document}

\begin{abstract}
We  use  a spectral  sequence to  compute    twisted  equivariant  $K$-Theory  groups for  the  classifying  space  of  proper  actions of  discrete groups.  We  study  a  form  of  Poincar\'e  Duality for twisted  equivariant  $K$-theory studied by  Echterhoff, Emerson and  Kim in the  context  of the  Baum-Connes  Conjecture  with coefficients and  verify  it for  the  Group  $\sldrei$.  
 
\end{abstract}
   \maketitle

\typeout{----------------------------  linluesau.tex  ----------------------------}

%%%%%%%%%%%%%%%%%%%%%%%%%%%%%%%%%%%%%%%%%%%%%%%%%%%%%%%%%%%%%%%%%%%%%%%%%%%%%%%%%
%%%%%%%%%%%%%%%%%%%%%%%%%%%%%%%%%%% Abstract  %%%%%%%%%%%%%%%%%%%%%%%%%%%%%%%%%%%%%%%
%%%%%%%%%%%%%%%%%%%%%%%%%%%%%%%%%%%%%%%%%%%%%%%%%%%%%%%%%%%%%%%%%%%%%%%%%%%%%%%%%

%\typeout{------------------------------------ Abstract ----------------------------------------}
In  this  work, we  examine  computational  aspects  relevant  to  the  computation  of  twisted  equivariant  $K$-theory  and  $K$-homology  groups for  proper  actions  of  discrete  groups. 

Twisted $K$-theory   was introduced  by   Donovan  and Karoubi   \cite{donovankaroubi}  assigning  to a  torsion  element  $\alpha \in H^{3}(X, \mathbb{Z})$ abelian  groups   $^ \alpha K^{*}(X)$ defined  on   a  space  by   using  finite  dimensional matrix  bundles. After  the  growing  interest  by  physicists  in  the  1990s  and  2000s, Atiyah and  Segal \cite{atiyahsegaltwisted} introduced  a   notion  of  twisted  equivariant  $K$-theory for  actions of compact Lie  Groups. In  another  direction, or\-bi\-fold versions  of  twisted  $K$-theory  were  introduced  by  Adem  and Ruan \cite{ademruan}, and   progress   was made  to  develop  computational  tools for  Twisted  Equivariant  $K$-Theory  with the  construction  of  a  spectral sequence in \cite{barcenasespinozauribevelasquez}. 

The  paper  \cite{barcenasespinozajoachimuribe} introduces  Twisted  equivariant  $K$-theory for  proper  actions, allowing a more  general class  of  twists, classified  by the  third  integral  Borel cohomology  group $H^3(X \times_G EG, \mathbb{Z})$.   
  
 We  concentrate  in  the  case  of twistings  given  by discrete  torsion, which  is  given  by    cocycles  
$$\alpha \in Z^2(G, S^1)$$

representing  classes  in the image  of  the projection  map 
$$H^2(G, S^1) \overset{\cong}{\to} H^3(BG, \mathbb{Z}) \to H^3(X{\times }_G EG, \mathbb{Z}).$$

Under this  assumption on the  twist, a  version  of  Bredon cohomology  with  coefficients  in  twisted  representations can  be  used  to approximate   twisted  equivariant  $K$-Theory, by  means of  a spectral  sequence   studied  in  \cite{barcenasespinozauribevelasquez} and \cite{dwyer2008}. 

The Bredon (co)-homology  groups  relevant  to  the  computation  of  twisted  equivariant  $K$- theory, and  its  homological  version, twisted  equivariant $K$-homology  satisfy  a  Universal  Coefficient  Theorem, \ref{theoremUCT}. We state it  more  generally for  a  pair of  coefficient systems  satisfying conditions  \ref{conditionD}.  

\begin{theoremn}[Universal Coefficient Theorem]
Let  $X$  be  a  proper,  finite    $G$-CW complex.  Let  $M^ ?$ and  $M_?$  be  a  pair  of functors  satisfying  Conditions \ref{conditionD}.  Then, there  exists   a short  exact  sequence  of  abelian  groups

 $$ 0\to {Ext}_{\mathbb{Z  }}  (H_{n-1}^{G} (X, M_?), \mathbb{Z})\to  H^ {n}_{G}(X,  M^?) \to {Hom}_{\mathbb{Z  }}  (H_{n}^{G } (X, M_?), \mathbb{Z}) \to  0 $$ 

\end{theoremn}

 The  main  application  for  the  cohomological  methods described in this note  are  computations   verifying  a  form  of  Poincar\'e Duality  in  twisted  equivariant  $K$-Theory for  discrete  groups, as  studied  by  Echterhoff,  Emerson and  Kim  \cite{echterhoffemersonkim} in the  context  of  the  Baum-Connes  Conjecture  with coefficients.

Given  a $G$-C*-algebra  $A$, the Baum-Connes  conjecture with coefficients  in  $A$  for  $G$  predicts  that  an  analytical  assembly  map 
$$K_*^G(\eub{G}, A)= KK_{G}(C_{0}(\eub{G}), A) \longrightarrow  K_{*}(A\rtimes G)$$
is  an  isomorphism,  where  $A\rtimes G$  is  the  crossed  product  $C^*$-algebra.

A  discrete  torsion  twist  defines  an action  of   $G$  on the  $C^*$- algebra   of  compact  operators  on  $l^2(G)$,  which only  depends  on the cohomology  class  of  $\omega$  and  not on particular representing  cocycles.  We  will  denote  this $G$-$C^*$-algebra by  $K_\omega$. 
Given a  discrete  torsion twist $\omega$, there  exist analytic  versions   of  twisted equivariant  $K$-homology  and $K$-theory groups, $K^*_G(X, \omega)$, $K_*^G(X,\omega)$,  which are defined  in terms  of  the  equivariant  Kasparov $KK$-theory groups
$KK_G^*(C_0(X), K_\omega)$, respectively $KK_G^*(K_\omega, C_0(X))$. See  \cite{echterhoffemersonkim} for more  details. 

Our  computational  methods  describe a  duality  relation, corollary  \ref{corollaryduality} between  these  groups as  follows:

\begin{corollaryn}[Duality for twisted equivariant K-theory]

Let  $G$  be  a  discrete  group  with a finite model  for  $\eub{G}$. Let  $\omega \in Z^2(G, S^1)$.  Assume  that the  Bredon  cohomology  groups  $H^{*}_G(X, R^{- \alpha})$ relevant  to  the  computation of the  twisted  equivariant $K$-theory are  all  free  abelian and are concentrated in degree 0 and 1.  Then,   there  exists  a  duality  isomorphism
$$K^{*}_G(\eub{G}, K_{-\omega}) \longrightarrow K^{*}_G( K_{\omega}, \eub{G})$$
\end{corollaryn}

The  family  of  groups $G$ satisfying  the previous asummptions  on  both  the  twist and  the classifying space  $\eub{G}$ includes several  examples. We  pay  attention  to  $\sldrei$. 

 We  use  the  computation  of  the  cohomology  of  $\sldrei$ due  to  Soul\'e \cite{Soule(1978)},  previous  work  by  S\'anchez-Garc\'ia \cite{Sanchez-Garcia(2006SL)},  as  well  as  the  theory  of  projective  representations  of  finite  groups to  compute  the  twisted  equivariant  $K$-theory  on  the  classifying space  $\eub{\sldrei}$,  by  computing  the Bredon  cohomology  group associated  to  a specific  torsion twist. Using  the  spectral  sequence  we verify  that  the  computation of  twisted  equivariant  $K$-theory  reduces  to  Bredon cohomology in Theorem \ref{theoremtwistedktheory}. 
 
\begin{theoremn}[Calculation of twisted equivariant K-homology of $\sldrei$]
The  equivariant $K$-homology  groups of  $\sldrei$ with  coefficients  in the  $\sldrei $-$C^*$  algebra  $\KK_{u_{1}}$  are  given  as  follows: 
\begin{equation*}
K_p^{\sldrei}(\eub{\sldrei},\KK_{u_{1}})= 0 \, \text{ p odd}, \quad K_p^ {\sldrei}(\eub{\sldrei}, \KK_{u_{1}})\cong\IZ^{\oplus13} \text{ p even}
\end{equation*}

\end{theoremn}
\subsection{Aknowledgments}
This  first  author  was  supported  by  the  Hausdorff Center for  Mathematics, Wolfgang L\"uck's Leibnizpreis and  a  CONACYT postdoctoral  fellowship.  The  second  author  was  supported  by COLCIENCIAS trough the grant ``Becas  Generaci\'on del  Bicentenario'' number  494, the  Fundaci\'on  Mazda  para  el  Arte y la  Ciencia,  Wolfgang L\"uck's  Leibnizpreis and  a  UNAM postdoctoral  fellowship.
The authors  thank  an anonymous  referee  for  valuable  suggestions  concerning  the presentation  of  the  material exposed  in this note.

\section{Bredon  cohomology }
We  recall  briefly some definitions  relevant  to  Bredon  homology and  cohomology, see \cite{valettemislin}  for  more  details. 
Let $G$ be a discrete group. A $G$-CW-complex is a CW-complex with a $G$-action permuting the cells and such that
if a cell is sent to itself, this is done by the identity map. We call the $G$-action proper if all cell stabilizers are finite
subgroups of $G$.
\begin{definition} A model for $\eub G$ is a proper $G$-CW-complex $X$ such that for any proper $G$-CW-complex $Y$ there is a
unique $G$-map $Y\rightarrow X$, up to $G$-homotopy equivalence.
\end{definition}
One can prove that a proper $G$-CW-complex $X$ is a model of $\eub G$ if and only if the subcomplex of fixed points $X^H$ is
contractible for each finite subgroup $H\subseteq G$.  It can be shown that classifying spaces for proper actions always exist. They are
clearly unique up to $G$-homotopy equivalence.

Let $\mathcal{O}_G$
be the orbit category of $G$; a category with one object $G/H$ for each subgroup $H\subseteq G$ and where  morphisms are  given  by  $G$-equivariant maps.  There  exists  a  morphism $\phi:G/H\rightarrow G/K$ if and only if $H$ is conjugate in $G$  to a subgroup of $K$.

\begin{definition}[Cellular Chain complex associated  to a  $G$-CW  complex]
Let  $X$  be  a  $G$-CW-complex.  The  contravariant   functor  $\underline{C}_{*}(X):\mathcal{O}_G\to  \mathbb{Z}-\CHAINCOMPLEXES  $  assigns  to every  object  $G/H$     the  cellular $\mathbb{Z}$-chain  complex   of  the $H$-fixed point  complex    $ \underline{C}_{*}(X^ {H})\cong C_{*}({\rm  Map  }_{G}(G/H, X))$  with  respect  to  the  cellular  boundary  maps $\underline{\partial}_{*} $. 
\end{definition}

We  will  use  homological  algebra  to  define Bredon  homology  and  cohomology  functors. 

A  contravariant  coefficient  system  is  a  contravariant  functor 
$$ M:\OO \to \mathbb{Z}-\MODULES$$

 % where  $\OO$  is   the  full  subcategory  of the  orbit  category  of  $G$,    $\mathcal{O}_{G, \mathcal{F}}$  generated   by  the  objects $G/H$ for  a  family of  subgroups   $H\in \mathcal{F}_G$.    

Given  a  contravariant  coefficient  system $M$, the  Bredon    cochain   complex \\ $C_G^*(X;M)$ is  defined  as the   abelian  group   of  natural  transformations   of  functors  defined  on  the  orbit  category $\underline{C}_{*}(X) \to  M$. In  symbols, 

$$C_G^n(X;M)=Hom_{\mathcal{O}_{\mathcal{F}_G}}(\underline{C}_n(X),M)$$
Where  $\mathcal{F}_G$ is  a  family containing  the  isotropy  groups  of  $X$. 

Given a  set  $\{e_{\lambda}\}$ of   orbit   representatives of  the n-cells of  the  $G$-CW  complex  $X$,  and isotropy  subgroups  $S_{\lambda}$  of  the  cells  $e_{\lambda}$,   the  abelian  groups $C_G^n(X,M)$  satisfy:
 
 $$C_G^n(X,M)= \underset{\lambda}{\bigoplus }Hom_{\mathbb{Z}}(\mathbb{Z}[e_{\lambda}], M(G/S_{\lambda}))$$   
 with  one  summand  for  each  orbit representative  $e_\lambda$.
They  afford  a differential $\delta^n:C_G^n(X,M)\to C_G^{n-1}(X,M)$ determined  by  $\underline{\partial}_*$ and  maps $M(\phi):  M(G/S_\mu)\to M( G/ S_\lambda )$  for  morphisms  $\phi:G/S_\lambda \to  G/S_\mu$.

\begin{definition}[Bredon cohomology] 
The  Bredon  cohomology  groups   with  coefficients  in  $M$, denoted  by  $H^{*}_{G} (X,  M)$    are  the  cohomology   groups  of  the  cochain  complex  $\big (C_{G}^ *(X, M), \delta^* \big )$.

 \end{definition}

Dually to  the  cohomological  situation,  given a  covariant functor  $$N:\OO\to  \mathbb{Z}-\MODULES,$$  the chain  complex 

$$ C_{*}^ G= \underline{C}^{n}(X) \underset{\OO}{\otimes}N =\underset{\lambda}{\bigoplus}\mathbb{Z}[e_{\lambda}]\otimes  N(G/S_{\lambda})$$ 

admits  differentials $\delta_* =\partial_* \otimes N(\phi)$  for    morphisms   $\phi:  G/ S_{\lambda}\to  G/S_{\mu}$ in $\OO$.

\begin{definition}[Bredon homology]
  The  Bredon  homology   groups  with  coefficients  in $N$ ,  denoted  by  $H_{*}^{G}(X, N)$.   are  defined  as  the  homology  groups   of  the  chain  complex $\big( C_{*}^{G}(X, N), \delta_* \big)$
\end{definition}

\begin{remark}[Determination  of Bredon (co)-homology  in practice]\label{matrix}
The  coefficient  systems  considered  in  this  note  yield chain  complexes,  respectively  cochain complexes    of   free  abelian  groups with  preferred  bases  to  compute  both  Bredon  homology  and  cohomology. 

Notice that if we have a complex of free abelian groups
$$\cdots\rightarrow\IZ^{\oplus n}\xrightarrow{f}\IZ^{\oplus m}\xrightarrow{g}\IZ^{\oplus k}\rightarrow\cdots$$
with $f$ and $g$ represented by matrices $A$ and $B$ for some fixed basis, then the homology at $\IZ^{\oplus m}$ is
$$ker(g)/im(f)\cong\IZ/d_1\IZ\oplus\cdots\oplus\IZ/d_s\IZ\oplus\IZ^{\oplus(m-s-r)},$$
where $r = rank (B)$ and $d_1,\ldots, d_s$ are the elementary divisors of $A$.
\end{remark}

Given a  discrete  group  $G$, the  complex  representation ring  defines  two  functors   defined  on the  subcategory    generated by  objects  $G/H$  with  $H$  finite.  

$${\calr^ ?  \quad \calr_?} $$

For  every  object $G/H$,  the   groups $R^?(G/H), R_?(G/H)$ agree  with   the  Gro\-then\-dieck  group  of  isomorphism  classes  of  complex   representations  $R_\mathbb{C}    (H)$  of  the finite   subgroup $H$. 

The   contravariant   functor $\calr^?$   assigns  to  a   $G$-map $\phi:G/H\to  G/K$  the   restriction  map  
$R_{\mathbb{C}}(K)\to  R_{\mathbb{C}}(gHg^-1)  \cong R_{\mathbb{C}}(H)$   to the  subgroup $gKg^{ -1}$ of  $H$  determined  by  the  morphism $\phi:G/H\to  G/K$, where $g\in G$  is  such  that $\phi(eH)=gK$.

The  covariant  functor $\calr_?$  assigns   the  the  morphism  $\phi$  the   induction  map  $R_{\mathbb{C}}(H) \cong  R_{\mathbb{C}}(gHg^{-1})\to   R_{\mathbb{C}}(K)$.

Given  a  finite  group $H$  the  group  $R_\mathbb{C}(H)$  is  free  abelian,  isomorphic  to  the  free  abelian  group  generated  by  the  set $\rho_{1}, \ldots, \rho_{s}$ of  irreducible  characters.  For any  representation $\rho$,  there  exists  a   unique  expression
$\rho= n_{1} \rho_{1}, \ldots, n_{s}\rho_{s}$, where $n_{i}= (\rho \mid \rho_{i})$, and  $(\quad \mid \quad)$  is  the scalar product  of  characters.

Recall  that  due  to  Frobenius  reciprocity,  given  a  subgroup  $K$ of $H$, a  representation $\tau$  of  $K$ and  a   representation $\rho$  of  $H$,  the  equation   $(\tau\uparrow\mid\rho)_H=(\tau\mid\rho\downarrow)_{K}$ holds, where $\downarrow$ denotes  restriction and  $\uparrow$  denotes  induction.

\subsection{Bredon  (co)-homology  with  coefficients  in twisted  representations. }
\begin {definition}
Let  $H$  be a   finite  group and  $V$  be  a  complex  vector  space. Given  a   cocycle     $\alpha:H\times H\to S^{1}$  representing  a  class  in  $ H^{2}(H,S^{1})\cong H^{3}(H,\mathbb{Z})$, an $\alpha$-twisted  representation  is  a   function $P:H\to Gl(V)$  satisfying: 
$$P(e)={\rm 1}$$  
$$P(x)P(y)=\alpha(x,y)P(xy)$$
\end{definition}
The  isomorphism  type  of an $\alpha$-twisted  representation  only  depends  on the  cohomology  class  in  $H^{2}(H,S^{1})$.

\begin{definition}
Let  $H$  be  a  finite  group  and   $\alpha:H\times H\to S^{1}$  be  a cocycle  representing a class  in  $ H^{2}(H,S^{1})\cong H^{3}(H,\mathbb{Z})$.  The  $\alpha$-twisted  representation   group of $H$, denoted  by  $^{\alpha}\mathcal{R}(H)$  is   the  Grothendieck  group  of  isomorphism classes  of complex, $\alpha$-twisted representations with direct  sum as binary  operation. 
\end{definition}

Let   $H$  be  a  finite  group.  Given  a   cocycle  $ \alpha \in H^{2}(H,S^{1})$ representing  a  torsion class of  order  $n$,  the   normalization  procedure gives  a  cocycle $\beta$  cohomologous  to  $\alpha$  such that  $\beta: H\times  H\to  S^ {1}$ takes  values  in  the  subgroup  $\mathbb{Z}/n\subset S^{1}$   generated  by  a  primitive  $n$-th  rooth  of  unity $\eta$.  Associated  to  a  normalized  cocycle, there  exists   a   central  extension  
$$1\to  \mathbb{Z}/n\to H^*\to H\to  1$$
with  the  property  that  any   twisted   representation  of $H$  is  a linear  representation  of  $H^*$, with   the  additional  property  that  $\mathbb{Z}/n$  acts  by   multiplication  with  $\eta$. Such a  group  is called  a  Schur covering  group for  $H$.

\begin{definition}
Let  $\rho:H\to  Gl(V)$  be  an $\alpha$-twisted  representation.  The  character  of  $\rho$  is  the  map  $H\to \mathbb{C}$  given  by  $\chi(h) = {\rm trace}(\rho(h))$. 
\end{definition}
Given  a cocycle  $\alpha$, an  element $h\in H$  is  said  to  be  $\alpha$-regular  if  $\alpha(h,x)= \alpha(x,h)$ for  all $x\in H$. 
For  a   choice  of   representatives  of  conjugacy  classes  of  $\alpha$-regular  elements, $\{x_{1}, \ldots , x_{k}\}$ in  $H$,  an  $\alpha$-character  table    gives  the  values  of  the  character  of  irreducible $\alpha$-twisted  representations  by  evaluating $\chi(x_{i})$. An $\alpha$-character  is  not  constant  in  conjugacy  classes  of  elements  in  $H$; it  depends  on the  choice  of representatives of the  conjugacy  classes  of $\alpha$-regular  elements.  The  following  result  shows  the  relation to  different  choices of a  set  of  representatives   by  comparing  it  to the  linear  characters  of a  Schur  covering  group $H^ *$.  It  is  proven  in Theorem 1.1, Part  i Chapter 5, page  205  in \cite{karpilovsky3}. 

\begin{theorem}[Rectification  procedure for  characters  of  a  Schur  Covering group]\label{theoremrectification}
Let $1\rightarrow A\rightarrow H^*\xrightarrow{f}H\rightarrow 1$ be a finite central group
extension and let $\mu:H\rightarrow H^*$ be a fixed section of $f$. For any given $\xi\in Hom(A,\IC^*)$, let $\alpha =\alpha_\xi\in Z^2(H,\IC^*)$ be defined by

$$\alpha(x,y)=\xi(\mu(x)\mu(y)\mu(xy)^{-1})\text{ for all }x,y\in H$$
Then if $\lambda_1^*,\ldots,\lambda_n^*$ are all distinct irreducible $\IC$-characters of $H^*$ whose restriction
to $A$ has $\xi$ as an irreducible constituent and if $\lambda_i: H\rightarrow\IC$ is defined by
$$\lambda_i(g):= \lambda_i^*(\mu(g)) \text{ for all } g\in H$$
then $\lambda_1,\ldots,\lambda_r$ are all distinct irreducible $\alpha$-characters of $H$. Moreover,  if $\mu $  is  a  conjugation  preserving  section, then each  $\lambda_i$ is  a  class  function. 
\end{theorem}

 We can define contravariant and   covariant  coefficient  systems  for  the  family  $\mathcal{F}_G= \calfin$  of  finite  subgroups    
 agreeing  on  objects  by using the $\alpha$-twisted representation group functor $\calr_\alpha(\quad)$. 
 
 \begin{definition}\label{definitiontwistedcoefficients}
 Let $G$  be  a  discrete  group and  let   $\alpha\in Z^ {2}(G,S^ {1})$ be  a  cocycle.  
 Define $\calr_\alpha$ on objects by
$$ \calr{_{\alpha}}_?(G/H)  =  \calr_\alpha^?(G/H): =^{i^ {*}(\alpha)}\calr (H)$$

where $ i:H \to G$  is  the  inclusion.
 
And  induction of  $\alpha$-twisted, representations   for  the covariant,  part $  \calr{_{\alpha}}_?  $,    respectively restriction  of $\alpha$-representations for  the  contravariant  part $\calr_\alpha^?$. 
 \end{definition}

Orthogonality  relations  between  irreducible $\alpha$-characters  and  Frobenius  reciprocity  go  over  the  the   setting  of   $\alpha$-twisted  representations, compare  Proposition  11.7  in  \cite{karpilovsky3}, page  73 and Theorem 11.8  in  page  73.

\begin{definition}
Let $G$  be  a  discrete  group, let  $X$  be  a  proper $G$-CW complex, and  let  $\alpha\in Z^{2}(G, S^ {1})$  be  a  cocycle. The  $\alpha$-twisted Bredon  cohomology,  respectively  $\alpha$-twisted  homology  groups  of  $X$  are  the  Bredon  cohomology, respectively  homology  groups  with  respect  to  the   functors  described  in definition \ref{definitiontwistedcoefficients}.

\end{definition}

We  will  consider coefficient  systems which  consist  of  abelian  groups  with  preferred bases. This  condition  produces  a convenient  duality  situation, and  produces  based (co)-chain complexes  as  input  for  the  computation of Bredon (co)-homology  groups.  
\begin{condition}\label{conditionD}
Let $G$ be  a  discrete  group, Let $ M_?$ and  $M^?$ be covariant,  res\-pec\-ti\-vely  contravariant  functors  defined  on a  subcategory $\OO$  of  the  orbit  category $\mathcal{O}$ agreeing  on  objects.  Suppose  that 

\begin{itemize}
\item There  exists  for  every  object $G/H$ a  choice of  a  finite  basis  $\{\beta_{i^ {H}}\}$    expressing $M_?(G/H)= M^ ?(G/H) $ as  the   finitely  generated,  free abelian  group on $\{\beta_{i^ {H}}\}$   and  isomorphisms  $a_{H}: M^{?}(G/H) \overset{\cong}{\rightarrow}\mathbb{Z} [ \{\beta_{i^ {H}}\}] \overset{\cong}{\leftarrow}M_{?}(G/H) :b_{H}$.

\item  For  the  covariant  functor  $\widehat{M}:=Hom_ {\mathbb{Z}}(M^ ?(\quad), \mathbb{Z})$,    the  dual  basis  $\{\widehat{\beta}_{i^ {H}}\}$ of  $Hom_\mathbb{Z}(\mathbb{Z} [ \{\beta_{i^ {H}}\}] , \mathbb{Z})$  and   the   isomorphisms $a_{H}$ and  $b_{H}$, the  following  diagram  is commutative: 
$$\xymatrix@d{ \widehat{M}(G/H) \ar[r]^{\widehat{M}(\phi)}  &  \widehat{M}(G/K) \\ \mathbb{Z}[\{\widehat{\beta}_{i^ {H}}\}] \ar[u]^{\widehat{a_{H}}} &   \mathbb{Z}[\{\widehat{\beta}_{j^ {K}}\}] \ar[u]_{\widehat{a_{K}}}  \\ \mathbb{Z} [ \{\beta_{i^ {H}}\}] \ar[u]^ {D_H}&  \mathbb{Z} [ \{\beta_{j^ {K}}\}] \ar[u]_{D_K}\\        M_{?}(G/H) \ar[r]_{M_?(\phi)} \ar[u]^{b_{H}}  & M_{?}(G/K) \ar[u]_{b_K} }$$

Where  $D_H$, $D_K$  are  the  duality  isomorphisms  associated  to  the  bases and   $\phi:G/H\to  G/K$ is  a  morphism  in the  orbit  category.

\end{itemize}
 
\end{condition}
Conditions  \ref{conditionD}  are  satisfied  in some  cases:

\begin{itemize}
\item Constant  coefficients $\mathbb{Z}$. 
\item  The  complex representation  ring  functors  defined  on the  family $\mathcal{FIN}$  of  finite  subgroups,  $\calr^ ?$, $\calr_?$. A  computation using  characters as  bases  and  Frobenius  reciprocity yields  conditions \ref{conditionD}.
\item Consider  a   discrete group $G$ and  a    normalized torsion cocycle 
 $$\alpha \in Z^2 (G,S^1),$$ take the  $\alpha$- and  $\alpha^{-1}$ twisted  representation  ring functors $ \calr{_{-\alpha}}_?$  $\calr{_{\alpha}}^ ?$ defined  on the  objects  $G/H$, where  $H$   belongs  to  the family $\calfin$ of  finite  subgroups. Consider  for  every  object  $G/H$  the  cocycles $i_{H}^{*}(\alpha)$, where  $i_{H}:H\to G$ is  the  inclusion, and  assume  without loss  of  generality  that  they  are  normalized and  correspond to a  family  of  Schur  covering  groups in central extensions   $1\to \mathbb{Z}/n_{H} \to H^* \to  H\to  1 $.  

We  select  the set  $\{\beta_{H}\}$  given   as the  set  of  characters of irreducible  representations  of   $H^*$  where  $\mathbb{Z}/n_{H}$  acts by  multiplication  with a primitive  $n_{H}$-th root  of  unity.  Given  a  choice   of  sections   for  the quotient maps  $H^ *\to H$, one  can  construct  isomorphisms $^{i^{*}(\alpha)} {\calr}      (G/H) \overset{\cong}{\rightarrow} \mathbb{Z} [ \{\beta_{H}\}] $. The  orthogonality  relations and Frobenius  reciprocity  for their  twisted characters guarantee that conditions \ref{conditionD} yield.

\end{itemize}

%a_{H}:
%\overset{\cong}{\leftarrow}\calr_{?}(G/H) :b_{H}$

\begin{theorem}\label{theoremUCT}
Let  $X$  be  a  proper,  finite    $G$-CW complex.  Let  $M^ ?$ and  $M_?$  be  a  pair  of functors  satisfying  conditions \ref{conditionD}.  Then, there  exists   a short  exact  sequence  of  abelian  groups

 $$ 0\to {Ext}_{\mathbb{Z  }}  (H_{n-1}^{G} (X, M_?), \mathbb{Z})\to  H^ {n}_{G}(X,  M^?) \to {Hom}_{\mathbb{Z  }}  (H_{n}^{G } (X, M_?), \mathbb{Z}) \to  0 $$ 

\end{theorem}
\begin{proof}
The  proof  consists  of  two  steps: 
\begin{itemize}
\item Construction of  chain homotopy  equivalences  
$$\big (C_n^ G (X, M_?) , \delta_n \big ) \to \big ( C_n^ G (X, \widehat{M}), \hat{\delta^n}\big)  $$
$$\big ( C^n_G(X,M^?), \delta^n \big )  \to \big (  (Hom_{\mathbb{Z}}(C_n^G(X, \widehat{M}), \mathbb{Z}) ; Hom_{\mathbb{Z}}(\delta_{n}, \mathbb{Z}) \big )$$

\item Construction of  an exact  sequence $$   0\to \ker h  \to  H^ {n}_{G}(X,  M^?) \overset{h}{\to} {Hom}_{\mathbb{Z  }}  (H_{n}^{G } (X, \widehat{M}), \mathbb{Z}) \to  0  $$ and  identification  of  the  kernel  as  ${\Ext}_{\mathbb{Z  }}  (H_{n-1}^{G} (X, M_?), \mathbb{Z})$.

\end{itemize}

For  the  first chain homotopy  equivalence,  notice  that    given a  set  $\{e_{\lambda}\}$ of   orbit   representatives of  the n-cells of  the  $G$-CW  complex  $X$,  and isotropy  subgroups  $S_{\lambda}$  of  the  cells  $e_{\lambda}$, the   chain  complex  for computing  the  $n$-th Bredon  homology  with  coefficients  in $M_?$  is  given  by  $\underset{\lambda}{\bigoplus}\mathbb{Z}[e_{\lambda}]\otimes  M(G/S_{\lambda})$.  Condition \ref{conditionD} gives  for  every $\lambda$ an  isomorphism of  free  abelian  groups   $ M_{?}(G/S_{\lambda})\to  \widehat{M}(G/S_{\lambda})$ which   gives  a  chain  map  
$$C_n^ G(X, M_?)\cong  \underset{\lambda}{\bigoplus}\mathbb{Z}[e_{\lambda}]\otimes  M_?(G/S_{\lambda})\longrightarrow \underset{\lambda}{\bigoplus}\mathbb{Z}[e_{\lambda}]\otimes  \widehat{M}(G/S_{\lambda})\cong C_n^ G(X, \widehat{M})$$ giving  a chain  homotopy equivalence  and  subsequently an  isomorphism in homology  groups $H_n^ G (X, M_?)\overset{\cong}{\to} H_n^ G (X, \widehat{M})$ . 

For  the  second  isomorphism,  consider  the  chain  map   $$A:C^n_G(X, M^ ?) \to Hom_{\mathbb{Z}}(C_n^G(X, \widehat{M}), \mathbb{Z})$$

  $$A:  Hom_{\mathbb{Z}}\big  ( \underset{\lambda}{\bigoplus} \mathbb{Z}[e_{\lambda}],M^?(G/S_{\lambda} ) \big )   \to Hom_{\mathbb{Z}}\big (  \underset{\lambda}{\bigoplus} \mathbb{Z}[e_{\lambda}] \otimes \widehat{M}(G/S_{\lambda}), \mathbb{Z}\big )  $$

 assigning  to  $\varphi \in C^n_G(X, M^ ?)$ the  homomorphism  $\psi$  defined  on  basis  elements $[e_{\lambda}]\otimes \widehat{\beta}_{i_{S_{\lambda}}}$ as $\widehat{\beta}_{i_{S_{\lambda}}}(\varphi([e_{\lambda}]))\in \mathbb{Z}$.
A chain homotopy   inverse  for  this  map  is  given  by  the   chain map  

$$B: Hom_{\mathbb{Z}}\big (  \underset{\lambda}{\bigoplus} \mathbb{Z}[e_{\lambda}] \otimes \widehat{M}(G/S_{\lambda}), \mathbb{Z}\big )  \to  Hom_{\mathbb{Z}}\big  ( \underset{\lambda}{\bigoplus} \mathbb{Z}[e_{\lambda}],M^?(G/S_{\lambda} ) \big )$$

assigning  to  a   homomorphism $\psi \in Hom_{\mathbb{Z}}\big (  \underset{\lambda}{\bigoplus} \mathbb{Z}[e_{\lambda}] \otimes \widehat{M}(G/S_{\lambda}), \mathbb{Z}\big )  $  the  homomorphism   defined  on the  basis  $[e_{\lambda}]$  as  $\psi([e_{\lambda}])=   \underset{i_{S_{\lambda}}}{\Sigma}\psi( [e_{\lambda}]\otimes \widehat{\beta}_{i_{S_{\lambda}}}) \beta_{i_{S_{\lambda}}}.$
 \\

For  the  second part,  consider  the cochain  complex given  in degree  $n$  by  
$$C_G^n(X;M^ ?)=\big (Hom_{\mathcal{F}_G}(\underline{C}_n(X),M) , \delta^n \big) $$ and  the  chain complex $C_{n}^ G(X, \widehat{M}) = \big (   \underline{C}^{n}(X) \underset{\OO}{\otimes} \widehat{M}   , \delta_n \big ) $. Denote  by  $Z^{n}= \ker{\delta^n}$, $B^{n}= \im{\delta^{n-1}}$ and  $Z_{n}=  \ker \delta_{n}$, $B_{n}=  \im  \delta_{n+1}$.  

A class in $H_n^ G (X, \widehat{M})$ is represented by a homomorphism $\varphi:C_n^ G (X, \widehat{M})\rightarrow \IZ$ such that $Hom_\IZ(\delta_n,\IZ)(\varphi)=0$, that means $\varphi\mid B_n=0$. Then if $\varphi_0=\varphi\mid Z_n$ we have a map defined in the quotient $\overline{\varphi}_0:Z_n/B_n\rightarrow\IZ$, or in others words an element of 
$Hom_\IZ(H_n^ G (X, \widehat{M}),\IZ)$. We have defined a map

$$h: H^ {n}_{G}(X, \widehat{M}) \to {Hom}_{\mathbb{Z  }}  (H_{n}^{G } (X, \widehat{M}), \mathbb{Z}).$$

This map is surjective. Now we proceed to identify $\ker(h)$.

As we have a projective resolution of $H_ {n}^{G}(X, \widehat{M})$
$$0\rightarrow B_n\rightarrow Z_n\rightarrow H_n^G(X, \widehat{M}),$$

the group $\Ext(H_ {n}^{G}(X, \widehat{M},\IZ)$ can be calculated using the exact sequence

\begin{equation}\label{resolution}Hom_\IZ(Z_n,\IZ)\rightarrow Hom_\IZ(B_n,\IZ)\rightarrow\Ext(H_ {n}^{G}(X, \widehat{M},\IZ)\rightarrow0.\end{equation}
 On the other hand,  we have the following long exact sequences
 \begin{equation}\label{ex1}
 0\rightarrow Z_n\rightarrow C_n^ G(X, \widehat{M})\xrightarrow{\delta}B_{n-1}\rightarrow0
 \end{equation}
 
\begin{equation}\label{ex2}
 0\rightarrow H_ {n}^{G}(X, \widehat{M})\rightarrow C_n^ G(X, \widehat{M})/B_n\xrightarrow{\delta}B_{n-1}\rightarrow0
 \end{equation}

The sequences \ref{ex1} and \ref{ex2}  are split because $B_{n-1}$ is free abelian.

We can decompose $\delta$ as the map 
\begin{multline}
\label{ex3}
 0\longrightarrow C_n^ G(X, \widehat{M}) \\ \longrightarrow C_n^ G(X, \widehat{M})/B_n\longrightarrow C_n^ G(X, \widehat{M})/Z_n \overset{\cong}{\longrightarrow} \\ B_{n-1}\subset Z_{n-1}\subset C_{n-1}^ G(X, \widehat{M})
 \end{multline}
We have the following diagram of exact sequences (here $H_n=H_ {n}^{G}(X, \widehat{M})$ and $C_n=C_n^{G}(X, \widehat{M})$).

\begin{equation}\label{ex4}
  $$\xymatrix{ &   & 0\\
   & & Hom_\IZ(H_ {n},\IZ)\ar[u]\\
  Hom_\IZ(B_n,\IZ) & Hom_\IZ(C_n,\IZ)\ar[l] & Hom_\IZ(C_n/B_n,\IZ)\ar[l]\ar[u] & 0\ar[l] \\
  0 & \Ext(H_ {n},\IZ)\ar[l]&Hom_\IZ(B_{n-1},\IZ)\ar[u]\ar[l] &  Hom_\IZ(Z_{n-1},\IZ)\ar[u]\ar[l]\\
  & & 0\ar[u] &  Hom_\IZ(C_ {n-1},\IZ)\ar[lu]\ar[u]
  }$$
\end{equation}

The right vertical sequence is exact because \ref{ex1} splits, the left vertical sequence is split exact because \ref{ex2} is. The upper horizontal sequence is obtained from applying $Hom$ to the projective resolution of 
$H_ {n}^{G}(X, \widehat{M})$ and the lower horizontal sequence is \ref{resolution}.

Analyzing the diagram \ref{ex4} and  decomposition \ref{ex3} we obtain

$$Z^n\cong Hom_\IZ(C_n(X, \widehat{M})/B_n,\IZ),$$

on the other hand $$B^n\cong im(Hom_\IZ(Z_{n-1},\IZ)\rightarrow Hom_\IZ(C_n(X, \widehat{M})/B_n,\IZ).$$
Then $$H^{n}_{G}(X,M^?)\cong\coker(Hom_\IZ(Z_{n-1},\IZ)\rightarrow Hom_\IZ(C_n(X, \widehat{M})/B_n,\IZ)).$$

The  maps 
$Hom_\IZ(Z_{n-1},\IZ)\to Hom_\IZ(B_{n-1},\IZ)\to Hom_\IZ(C_n(X, \widehat{M})/B_n,\IZ)$  together  with the  section 
$Hom_\IZ(C_n(X, \widehat{M})/B_n,\IZ)\to Hom_\IZ(B_{n-1},\IZ)$  induce  an exact  sequence  of cokernels  identifying  $\ker{h}$  with $Ext_\IZ(H_{n-1}^G(X, M_?), \IZ)$.

\end{proof}
\section{ Spectral  sequences  for  Twisted  Equivariant  $K$-Theory. }

Twisted  equivariant  $K$-theory  for  proper  and  discrete  actions   has  been  defined  in a  variety  of  ways. For  a  torsion  cocycle  $\alpha \in Z^{2}(G,  S^ {1})$, it is possible to define twisted equivariant $K$-theory  in  terms  of  finite  dimensional, so  called  $\alpha$-twisted  vector  bundles, an $\alpha$-twisted  equivariant  $K$-theory  for  proper  actions of  discrete  groups  on finite, proper $G$-CW  complexes.

\begin{definition}
Let  $\alpha\in Z^ {2}(G,S^1)$  be a  normalized  torsion cocycle  of order  $n$   for  the  discrete  group  $G$,  with associated  central  extension  $0\to  \mathbb{Z}/{n} \to  G_{\alpha}\to  G$.  An  $\alpha$-twisted  vector  bundle  is  a  finite  dimensional $G_{\alpha}$-  equivariant  complex vector bundle  such  that  $\mathbb{Z}/n$  acts  by  multiplication  with a  primitive  root  of  unity.  The   groups $^{ \alpha }\mathbb{K}^{0}_{G_\alpha}( X)$ are  defined  as   the  Grothendieck  groups  of  the  isomorphism classes of  $\alpha$-twisted  vector  bundles over  $X$.
  
\end{definition}
 Given  a proper  $G$-CW  complex $X$,  define  the  $\alpha$-twisted  equivariant  $K$-theory  groups $^{\alpha}K^{-n}_{G}(X)$  as  the  kernel  of  the  induced  map 

$$ ^{ \alpha }\mathbb{K}^{0}_{G_\alpha}( X\times S^{n})   \overset{{\rm incl}^{*}}{\to}       {  ^{ \alpha }\mathbb{K}^{0}_{G_\alpha}( X)} $$
 
The  $\alpha$-twisted  twisted  equivariant  $K$-theory   catches  relevant  information  to  the class  of  twistings coming from  the  torsion  part  of  the  group  cohomology  of  the  group, in the  sense  that  the $K$-groups  trivialize  for  cocycles  representing non-torsion  classes.

As  noted in \cite{barcenasespinozauribevelasquez}, there is a spectral sequence connecting the $\alpha$-twisted Bredon cohomology and the $\alpha$-twisted
equivariant K-theory of finite proper $G$-CW complexes. When the twisting is discrete this spectral sequence is a special case of the Atiyah-Hirzebruch spectral sequence  for \emph{untwisted} $G$-cohomology theories constructed by Davis and L\"uck \cite{davisluck}. In particular, it  collapses  rationally.
\begin{theorem}\label{spectral} Let $X$ be a finite proper $G$-CW complex for a discrete group $G$, and let $\alpha\in Z^2(G,S^1)$ be a
normalized torsion cocycle. Then there is a spectral sequence with

$$E_2^{p,q}=\begin{cases}
  H^p_G(X,\calr_{\alpha}^?)&\text{ if $q$ is even}\\
                     0 &\text{ if $q$ is odd}
\end{cases}$$
so that $E_\infty^{p,q}\Rightarrow { }^\alpha K^{p+q}_G(X)$.
\end{theorem}

In \cite{barcenasespinozajoachimuribe},  a  definition  for  twisted  equivariant  $K$-Theory  is  proposed, where   the  class  of  twistings  is  extended  to   include  arbitrary  elements  in the  third Borel cohomology  group with  integer  coefficients $H^{3}(X{\times}_G{\rm E} G, \mathbb{Z})$. Extending  the  work  by \cite{dwyer2008}, this  theory  gives  non trivial twisted  equivariant  $K$-theory  groups  for cocycles which  are  non  torsion.  In  the  work \cite{barcenasespinozauribevelasquez},  a   spectral  sequence  is  developed  to   compute  the  twisted  equivariant  $K$-theory  under  these conditions,  in terms  of  generalizations  of  Bredon  cohomology  which capture  more  general  twisting data.  Specializing to  the trivial group $\{e\}$, the spectral sequence of \cite{barcenasespinozauribevelasquez} yields explicit descriptions of both the  $E_2$ term and  the  differentials of the non-equivariant  Atiyah-Hirzebruch spectral sequence  described in \cite{atiyahsegalcohomology}. In contrast  to  the  spectral  sequence  constructed  in \cite{dwyer2008},  the  spectral  sequence   in \cite{barcenasespinozauribevelasquez}  does  not  collapse  rationally in general.

\section{Cohomology  of  $\sldrei$ and twists  }\label{sectiontwists}
We  concentrate  now  in the  example  of  the  group  $\sldrei$. This  group  is  particularly  accessible  due  to the  existence  of  a convenient  model  for the  classifying  space  of  proper  actions  and  the   fact  that  the  group  cohomology  relevant  to the  twists is  completely  determined  by  finite  subgroups. 
 
\subsection{Twist  Data  }
We  will  describe  the  twist data, which  reduce completely  to  torsion  classes.  After  the  work  of  Soul\'e, Theorem 4, page  14  in \cite{Soule(1978)} the   integral  cohomology  of  $\sldrei$ only  consists  of  $2$ and  $3$-torsion.  The  3-primary  part  is  isomorphic  to  the  graded   algebra 
$$\mathbb{Z}[x_{1}, x_{2}]\mid  3x_{1}=3x_{2}=0$$     
 with  both  generators  in  degree 4. 

The  two-primary  component  is  isomorphic  to  the   graded  algebra  
$$\mathbb{Z}[u_{1}, \ldots, u_{7}]$$
with    respective degrees 3,3,4, 4,5, 6,6, subject to the   relations 
$$2u_{1}=2u_{3}=4 u_{3}=4u_{4}=2u_{5}=2u_{6}=2u_{7}=0$$
$$\so{7}{} \so{1}{} = \so{7}{}\so{4}{}=\so{7}{}\so{5}{}=\so{7}{}\so{6}{}= \so{2}{} \so{5}{}= \so{2}{} \so{6}{}=0 $$
$$\so{7}{2}+\so{7}{}\so{2}{2}=\so{3}{}\so{4}{}+\so{1}{} \so{5}{}=\so{3}{} \so{6}{}+\so{3}{}\so{1}{2}=\so{3}{}\so{6}{}+ \so{5}{2}=0$$
$$\so{1}{} \so{6}{}+\so{4}{}\so{5}{}=\so{3}{0} \so{4}{2}+\so{6}{2}=\so{5}{} \so{6}{}+ \so{5}{} \so{1}{2}=0$$

The  twists  in  equivariant  $K$-theory  are  given  by  the  classes  in  $H^{3}(\sldrei, \mathbb{Z})$.  For  this  reason  we  shall  restrict  to  the  two-primary  component  in the  integral  cohomology. In  order  to  have  a  local  description  of  these  classes,  we   describe   the cohomology  of  some  finite  subgroups  inside $\sldrei$. Again Theorem 4 in page 14 of \cite{Soule(1978)}, gives the following result:
There exists an exact sequence of abelian groups ($n\in \IN$)
$$0\rightarrow H^n(\sldrei)_{(2)}\xrightarrow{\phi}H^n(S_4)_{(2)}\oplus H^n(S_4)_{(2)}\oplus H^n(S_4)_{(2)}\xrightarrow{\delta} H^n(D_4)\oplus H^n(\IZ_2)\rightarrow 0$$
where $\phi$  and $\delta$ (see     Corollary  2.1.b in  page  9  of  \cite{Soule(1978)})  are determined  by the system of inclusions

  $$\xymatrix{&&\sldrei\\S_4\ar[urr] & & S_4\ar[u]& & S_4\ar[ull]\\
  & D_4\ar[ul]^{i_2}\ar[ur]_{i_1}& & C_2\ar[ul]^{j_1}\ar[ur]_{j_2}\\
    } $$
  
Denote $R= i_1^*(H^*(S_{4}))\cap   i_2^*(H^*(S_{4}))$. Then, the image of the morphism $\phi: H^*(\sldrei)_{(2)}\rightarrow H^*(S_4)_{(2)}\oplus (i_1^*)^{-1}(R)$, is the set of elements $(y, z)$ such that 
  $j_2^*(y) = j_1^*(z).$ From  Soul\'e's work, we know that $H^*(S_4)_{(2)} = \IZ[y_1, y_2,y_3]$, with 
  $2y_1 = 2y_z = 4y_3 = y_1^4+ y_2^2y_1 + y_3y_1^2= 0$, and, if $R$ is as above, then
  $(i_1^*)^{-1}(R)=\IZ[z_1, z_2, z_3]$, with $2z_1 = 4z_2 = 2z_3 = z_3^2 + z_3z_1^2 = 0$. Furthermore $j_2^*(y_1) = t$, $j_2^*(y_2) = 0,$
$j_2^*(y_3) = t^2$, $j_1^*(z_1) = 0$, $j_1*(z_2) = t^2$, and $j_1^*(z_3) = 0$. Then the elements $u_1 = y_2$, $u_2 = z_1$, $u_3 =y_1^2+ z_2$, $u_4 = y_1^2+ y_3$, $u_5 = y_1y_2$,  $u_6 = y_1y_3 + y_1^3$ and $u_7 = z_3$ generate $H^*(\sldrei)_{(2)}$.
  
  In $H^3(\ )$ the above discussion can be summarized in the following diagram

  \begin{equation}\small
  \xymatrix{&  \langle u_1,u_2\rangle=H^3(\sldrei)\ar[dl]^{i^*}\ar[d]^{i^*}\ar[dr]^{i^*}\ar[drr]^{i^*}\\
  \langle z_1\rangle\subseteq H^3(S_4)\ar[d]^{i_1^*} &\langle z_1\rangle\subseteq H^3(S_4)\ar[dl]^{i_2^*}\ar[d]^{j_1^*} & \langle y_2\rangle \subseteq H^3(S_4)\ar[dl]^{j_2^*}
  \ar[d]&\langle y_2\rangle \subseteq H^3(D_6)\ar[dl]\\
   \langle x_3\rangle \subseteq H^3(D_4)\ar[d]& 0 & \langle y_2\rangle \subseteq H^3(D_2)\\
   \langle x_3\rangle \subseteq H^3(D_2)
  }
  \end{equation}
  There are  four twists up to cohomology, namely $0,u_1,u_2,u_1+u_2$, we will work with $u_1$.
\subsection{A  model  for  the  classifying space  of  proper actions}

We  recall  the  model  for  the  classifying space  for  proper  actions of  $\sldrei$,  as  described  in \cite{Soule(1978)},  but  also  in  
\cite{Sanchez-Garcia(2006SL)}. Let  $Q$  be  the   space  of real, positive   definite $3\times  3$-square   matrices. Multiplication  
by  positive  scalars  gives  an  action  whose  quotient  space  $Q/ \mathbb{R}^{+}$  is   homotopy  equivalent  to  $\sldrei /\eub{\sldrei}$. 

We  describe its  orbit  space. Let  $C$  be  the  truncated  cube   of  $\mathbb{R}^{3}$  with  centre  $(0,0,0)$ and  side  length 
$2$,  truncated  at  the  vertices   $(1,1,-1), (1, -1,1),(-1,1,1)$  and  $(-1,-1,-1)$,  trough  the  mid-points  of  the   
corresponding sides. As  stated  in \cite{Soule(1978)}, every matrix  $A$  admits  a  representative of  the  form   

\[ \left( \begin{array}{ccc}
2 & z & y \\
z & 2 & x \\
y & x & 2 \end{array} \right)\] 

which  may  be  identified  with  the  corresponding  point  $(x,y,z)$  inside  the  truncated  cube. We  introduce  the  following  
notation  for  the  vertices of  the  cube: 
\[ \begin{array}{ccc}
 O=(0,0,0) & Q=(1,0,0)\\
M=(1,1,1) & N=(1,1,1/2)\\
M^{'}=(1,1,0) & N^{'}= (1,1/2, -1/2)\\
P=(2/3,2/3, -2/3) & \\
\end{array}
\]

\begin{figure}[h]\label{figure1}
\caption{Triangulation  for  the  fundamental  region.}
\includegraphics[scale=15]{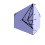}

\end{figure}

Note that the elements of $\sldrei$  
$$q_1=\begin{pmatrix}1&0&0
                                    \\0&0&-1
                                    \\0&1&1\end{pmatrix}
    q_2=\begin{pmatrix}-1&0&0
                                    \\0&1&1
                                    \\0&0&-1\end{pmatrix}$$send  the  triangle  $(M,N,Q)$  to the  triangle  $(M^{'}, N^{'} Q)$ and  the  quadrilateral 
 
$(N,N^{'},M^{'},Q)$  to $(N^{'}, N, M^{'}, Q)$. Thus, the  following identification  must  be  performed in the  quotient: 
$M\cong M^{'}$, $N\cong N^{'}$, $QM\cong  QM^{'}$, $QN\cong QN^{'}$, $MN\cong  M^{'}N^{'}\cong M^{'}N$ and   
$QMN\cong QM^{'}N\cong QM^{'}N^{'}$.

Following \cite{Sanchez-Garcia(2006SL)} we now describe the orbits of cells and corresponding stabilizers. This can be found also in Theorem 2 of Soul\'e's article \cite{Soule(1978)}
(although we use a cellular structure instead of a simplicial one). We have changed
the chosen generators so that they agree with the presentations on section 4. We summarize the information on
Table 1. We use the following notations: $\{1\}$ denotes the trivial group, $C_n$ the cyclic group
of $n$ elements, $D_n$ the dihedral group with $2n$ elements and $S_n$ the Symmetric group of permutations on $n$ objects.
\vspace{0.5cm}
\begin{center}
\begin{tabular}{ccccccccc}
\hline
vertices & & & & &2-cells\\
\hline
$v_1$ & $O$ & $g_2$, $g_3$ & $S_4$& &$t_1$ & $OQM$ & $g_2$ & $C_2$\\
$v_2$ & $Q$ & $ g_4$, $g_5$& $D_6$&&$t_2$ & $QM'N$ & $g_1$ & $\{1\}$\\
$v_3$ & $M$ &  $g_6$, $g_7$&  $S_4$& &$t_3$ & $MN'P$ & $g_{12},g_{14}$ & $C_2\times C_2$\\
$v_4$ & $N$ & $ g_6$, $ g_8$ & $ D_4$&& $t_4$ & $OQN'P$ & $g_5$ & $C_2$\\
$v_5$ & $P$ & $g_5$ , $g_9$ & $ S_4$ && $t_5$ & $OMM'P$ & $g_6$ & $C_2$\\
\hline
edges&&&&& 3-cells\\
\hline
$e_1$ & $OQ$ & $g_2$, $g_5$ &  $C_2\times C_2$ && $T_1$& $g_1$ & $\{1\}$\\
$e_2$ & $OM$ & $g_6, g_{10}$ & $D_3$\\
$e_3$ & $OP$ & $ g_6, g_5$ & $D_3$\\
$e_4$ & $QM$ & $g_2$ & $C_2$\\
$e_5$ & $QN'$ & $g_5$ & $C_2$\\
$e_6$ & $MN$ & $g_6, g_{11}$ & $C_2\times C_2$\\
$e_7$ & $M'P$ & $g_6, g_{12}$ & $ D_4$\\
$e_8$& $ N'P$ & $ g_5, g_{13}$ & $ D_4$\\

\hline

\end{tabular}
\end{center}
\vspace{0.5cm}

The first column is an enumeration of equivalence classes of cells; the second lists a representative of each class;
the third column gives generating elements for the stabilizer of the given representative; and the last one is the
isomorphism type of the stabilizer.  The
generating elements referred to above are

$$g_1=\begin{pmatrix}1&0&0\\
                                    0&1&0\\
                                    0&0&1\end{pmatrix}                    
     \hspace{1cm}g_2=\begin{pmatrix}-1&0&0\\
                                    0&0&-1\\
                                    0&-1&0\end{pmatrix}                               
      \hspace{1cm}g_3=\begin{pmatrix}0&0&1\\
                                    0&1&0\\
                                    -1&0&0\end{pmatrix}$$
                                    
$$g_4=\begin{pmatrix}-1&0&0\\
                                    0&1&1\\
                                    0&0&-1\end{pmatrix}                    
     \hspace{1cm}g_5=\begin{pmatrix}-1&0&0\\
                                    0&0&1\\
                                    0&1&0\end{pmatrix}                               
      \hspace{1cm}g_6=\begin{pmatrix}0&-1&0\\
                                    -1&0&0\\
                                    0&0&-1\end{pmatrix}$$
                                    
$$g_7=\begin{pmatrix}0&0&-1\\
                                    -1&0&0\\
                                    1&1&1\end{pmatrix}                    
     \hspace{1cm}g_8=\begin{pmatrix}-1&0&0\\
                                    0&1&0\\
                                    0&-1&-1\end{pmatrix}                               
      \hspace{1cm}g_9=\begin{pmatrix}0&0&-1\\
                                    -1&0&-1\\
                                    0&1&1\end{pmatrix}$$
                                    
$$g_{10}=\begin{pmatrix}0&0&-1\\
                                    0&-1&0\\
                                    -1&0&0\end{pmatrix}                    
     \hspace{1cm}g_{11}=\begin{pmatrix}-1&0&0\\
                                    0&-1&0\\
                                    1&1&1\end{pmatrix}                               
      \hspace{1cm}g_{12}=\begin{pmatrix}0&-1&-1\\
                                    0&-1&0\\
                                    -1&1&0\end{pmatrix}$$
                                    
$$g_{13}=\begin{pmatrix}0&1&1\\
                                    1&0&1\\
                                    0&0&-1\end{pmatrix}                    
     \hspace{1cm}g_{14}=\begin{pmatrix}-1&0&0\\
                                    -1&0&-1\\
                                    1&-1&0\end{pmatrix}$$                                    
                                    
Finally we describe the cells coboundary, we fix an orientation; namely, the ordering of the
vertices $O < Q < M < M' < N < N' < P$ induces an orientation in $E$ and also in   $\underline{B}\sldrei = E/\equiv$. We  resume  this  in  Figure 1.

\section{Representation theoretical  input}\label{sectionrepresentation}
\subsection{Cyclic  group $C_{2}$}  
The  cyclic  group    has  no  third  dimensional  integer  cohomology,  hence  the twisted  representation  ring   coincides  with the  usual one. 
The   character  table  is  as  follows:
\[\begin{array}{ccc}
C_{2}& 1 & g_{i} \\   
\rho_{1} & 1 &  -1\\ 
\rho_{2} & 1 & -1
  \end{array}
\]

\subsection{Dihedral  group  $D_{n}=\langle g_{i},  g_{j}\rangle =\langle g_{i}, g_{j} \mid g_{i}^{2}=g_{j}^{2} =(g_{i}g_{j})^{n}=1\rangle$ }   

The  dihedral group   of  order   six  has  only   the trivial  class  in  3  dimensional  integer  cohomology.  Thus  the  projective   representations  do  agree  with  the  linear  ones. In  the  even case,  the  subgroups  inside  $\sldrei$  are  $C_{2}\times  C_{2}=D_{2}$ and  $D_{4}$. 

The  following  is  the linear   character  table for  $D_{n}$: 
  
\[\begin{array}{cccc}
D_{n}& \langle (g_{i}, g_{j})^{k}\rangle  & \langle g_{j}(g_{i} g_{j})^{k} \rangle \\ 
\xi_{1} & 1 & 1\\   
\xi_{2}  & 1 &  -1\\ 
\hat{\xi_{3}} &  -1^{k}& -1^{k}\\
 \hat{\xi_{4}} & -1^{k} & -1^{k+1}\\
\hat{\phi_{p}} & 2\cos(2 \pi p  k/n) &  0 
  \end{array}
\]

where $0\leq k\leq n-1$,  p  varies  from  1  to  $(n/2) -1$ ( $ n$  even)  or  $(n-1)/2$ ( $n$ odd) and  the  hat  denotes  a   representation which  only  appears  in the   case  n  even. 
Most  of the information  is  taken  from  \cite{karpilovsky3},  Chapter  5  Section  7,  specially  Theorem  7.1  and   Corollary  7.2  in pages   258-261.  The  group $H^{2}(D_{n}, S^{1})$  is  isomorphic  to  $\mathbb{Z}/2$ if   n  is  even  and  $0$  if  $n$ is  odd. From  now  on  we  concentrate  in the  case  $n$  even. 
For  simplicity  consider  the  following  presentation  

$$D_{n}=\langle a,b\mid a^{n}=1, b^{2}=1, bab^{-1}=a^{-1}\rangle $$

Given  a  primitive  $nth$  root  of  unity  $\epsilon \in S^{1}$,  one  can  normalize  a   nontrivial  cocycle $\alpha: D_{n}\times D_{n}\to S^{1}$  to   one  satisfying  $\alpha(a^{i},a^{j}b^{k})=1$  and  $\alpha(a^{i}b, a^{j}b^{k})=\epsilon^{j}$. 
For $r\in \{1,2\ldots, n/2\}$  let  

\[ A_{r}=\left( \begin{array}{cc}
                 \epsilon^{r} & 0\\
                  0 & \epsilon^{1-r}
                \end{array}\right) \]
 and  
\[ B_{r}= \left( \begin{array}{cc}
          0& 1\\
	  1&0
         \end{array}\right) \]

Then, the  irreducible  $\alpha$-twisted  representations of  $D_{n}$  are  given  by   $\rho_{r}(a^{i}b_{j})=A^{i}_{r} B^{j}_{r}$  for  $i\in \{0,\ldots, n-1\}$   $j=0,1$,  and  $r\in \{1,\ldots,n/2\}$. The  projective representations $\rho_{j}$ for   $j\in \{1,\ldots, \frac{n}{2}\}$ are  nonequivalent  irreducible  projective  representations.

Consider  the  group 

$$D_2^*=\langle h_1, h_3, z\rangle$$
which is  isomorphic  to the  quaternions. 
A  linear  character  table  is  given  by   
\[
\begin{array}{cccccc}
 D_2^*& 1   & z   & \{h_1, h_1^{-1}\}&  \{h_3,h_3^{-1}\} &\{h_1h_3, (h_1h_3)^{-1}\} \\
  \eta_{1}     &   1        &   1           &    1     &  1       &    1         \\
  \eta_{2}     &   1        &   1           &   1     &  -1       &    -1          \\
  \eta_{3}     &   1        &   1           &    -1     &  1       &   -1          \\
  \eta_{4}     &   1        &   1           &    -1     &  -1      &    1         \\
  \eta_{5}     &   2        &   -2           &   0     &  0     &    0          \\
\end{array}
\]

The linear  character table for a Schur covering  group of $D_6$ is

 \[
\begin{array}{cccccccccc}
 D_{6}^*            &   1        &   z           &    a      &  za      &  a^2     &  za^2     &   a^3  & b     &    ab\\
  \gamma_{1}     &   1        &   1           &    1     &  1       &    1        &        1    &     1     &  1    &     1      \\
  \gamma_{2}     &   1        &   1           &    1     &  1       &    1        &        1    &     1     &  -1   &     1  \\
  \gamma_{3}     &   1        &   1           &   -1     & -1       &    1        &        1    &     -1    &   1  &     -1   \\
  \gamma_{4}     &   1        &   1           &    -1    &  -1      &    1        &       1     &     -1    &   -1 &      1   \\
  \gamma_{5}     &    2        &   2           &   -1     &  -1     &    -1       &       -1    &     -2    &   0  &      0   \\
  \gamma_{6}     &   2        &   2           &    -1     &   -1    &    -1       &       -1    &  {2}     &   0  &      0\\
  \gamma_{7} &2&-2&\frac{3+i\sqrt{3}}{2}&-\frac{3+i\sqrt{3}}{2}&\frac{1+i\sqrt{3}}{2}&-\frac{1+i\sqrt{3}}{2}&0&0&0\\
  \gamma_{8}     &   2      &  -2 &    0     &    0     &   -1-i\sqrt{3}&1+i\sqrt{3}&    0      &    0&0    \\
  \gamma_{9}  &   2      &  -2 &-\frac{3+i\sqrt{3}}{2} & \frac{3+i\sqrt{3}}{2}&\frac{1+i\sqrt{3}}{2}& -\frac{1+i\sqrt{3}}{2}  &    0      &    0&0    \\

\end{array}
\]

\subsection{Symmetric  group $S_{4}$ }\label{s4}
The  projective  representation  theory  of  the  Symmetric  groups  goes  back  to the  foundational  work  of  Schur \cite{schur}. The  information  concerning  the  representation  theory  of  the  Symmetric  group  in  four  letters  is  taken  from  \cite{Hoffman-Humphreys}, p.  46  and  \cite{karpilovsky3}, pages  215-243 . 
Recall  that  the  conjugacy  classes  inside  the   group  $S_{4}$ are  determined  by  their  cycle   type. The  cycle  type  of  a  permutation $\pi$  is  a  sequence  $(1^{a_{1}},2^{a_{2}}, \ldots, k^{a_{k}})$, where  the  cycle  factorization  of $\pi$  contains  $a_{i}$-cycles  of length  $i$. 

The  symmetric  group   admits  the  presentation 
$$ S_{4}= \langle  g_{1}, g_{2}, g_{3} \mid g_{i}^2={(g_{j}g_{j+1})}^3= {(g_{k}g_{l})}^2=1\rangle$$ 
 
$$1\leq i\leq 2, j=1, k\leq l-2$$
where  $g_{i}$ is the  transposition given by $(i-1, i)$. 

The  linear   character  table  of  $S_{4}$  is  as  follows: 
\[\begin{array}{cccccc}
   S_{4}       &   (1^{4}) & (2, 1^{2}) & (3,1)  & (4)   & (2^{2})\\
\theta_{1}     &    1      &     1     &    1    &   1   &   1   \\
\theta_{2}     &   1      &    -1      &    1    &   -1  &  1 \\
\theta_{3}     &   2      &     0      &    -1    &  0   &   2  \\
\theta_{4}     &   3      &      1     &     0   &   -1  &  -1 \\
\theta_{5}     & 3        &     -1     &    0   &   1    &   -1 \\
\end{array}
\]
The  representations  $\theta_{1}$  and  $\theta_{2}$  are  induced  from  the 1  dimensional trivial, respectively  the  sign  representation.  $\theta_{3}$  is  obtained from the  2-dimensional representation of  the  quotient  group $S_{3}$  the  character  $\theta_{4}$  is  given  as  the   $\xi- \theta_{1}$,  where  $\xi$  is  induced  from  the  $S_{3}$-trivial  representation, and   the  character  $\theta_{5}$  is the  character  assigned  to the  representation $g\mapsto\theta_{2}(g)V_{4}(g)$, where  $V_{4}$  is  the  irreducible  representation  associated  with  the  character $\theta_{4}$.    
The linear  character  table  of a  Schur  covering   group ${S_{4}^*}$    is  obtained  in page  254  of  \cite{karpilovsky3} by considering  the  group   with the  presentation
$$S_4^*=\langle h_1,h_2,h_3,z\mid h_i^2=(h_jh_{j+1})^3=(h_kh_l)^2=z, z^2=[z,h_i]=1\rangle$$ 
 $$1\leq i\leq 3,  j= 1, k\leq l-2$$

 and  the  central  extension 
$$1\to \langle z\rangle \to S_4^*\overset{f}{\to} S_4\to 1$$
given by $f(h_i)=g_i$, as  well as the  choice  of  representatives   of  regular conjugacy  classes  as  below.     
\[
\begin{array}{ccccccccc}
 S_{4}^*& (1^{4})   & (1^{4})^{'}   & (2,1^{2})&  (2^{2}) &  (3,1)  &  (3,1)^{'} &  (4)  & (4)^{'}\\
  \epsilon_{1}     &   1        &   1           &    1     &  1       &    1        &        1    &     1     &  1      \\
  \epsilon_{2}     &   1        &   1           &   -1     &  1       &    1        &        1    &     -1    &  -1     \\
  \epsilon_{3}     &   2        &   2           &    0     &  2       &   -1        &       -1    &     0     &   0     \\
  \epsilon_{4}     &   3        &   3           &    1     &  -1      &    0        &       0     &     -1    &   -1    \\
  \epsilon_{5}     &   3        &   3           &   -1     &  -1      &    0        &       0     &     1     &   1     \\
  \epsilon_{6}     &   2        &  -2           &    0     &   0      &   1         &       -1    &  \sqrt{2} &-\sqrt{2}\\
  \epsilon_{7}     &   2        &  -2           &    0     &   0      &   1         &       -1    & -\sqrt{2} & \sqrt{2}\\
  \epsilon_{8}     &   4        &  -4           &    0     &    0     &   -1        &       1     &    0      &    0    \\

\end{array}
\]
 where  the   first  five   lines  are  characters   associated  to  $S_{4}$, $\epsilon_{6}$  is  the  \emph{Spin } representation.

\section{Untwisted Equivariant  $K$-Theory.}\label{sectionuntwisted}

In this section we use the Atiyah-Hirzebruch spectral sequence to calculate the equivariant K-theory groups $K^*_{\sldrei}(\underbar{E}\sldrei)$. 
The cochain complex associated to the $\sldrei$-CW-complex structure of $\underbar{E}\sldrei$ described in section \ref{sectiontwists} is:

$$0\rightarrow\bigoplus_{k=1}^5R(stab(v_k))\xrightarrow{\Phi_1}\bigoplus_{j=1}^8R(stab(e_j))\xrightarrow{\Phi_2}
\bigoplus_{i=1}^5R(stab(t_i))\xrightarrow{\Phi_3} R(stab(T))\rightarrow0.$$
Where the $\Phi_i$ is the coboundary given by restriction over representations rings. 

Frobenius reciprocity implies that if $H'$ is a subgroup of $H$ and $$\uparrow_H'^H:R(H')\rightarrow R(H)$$
is represented by a matrix $A$ (in the basis given by irreducible representations), then $$\downarrow_H'^H:R(H)\rightarrow R(H')$$
is represented by the  transpose  matrix  $A^T$ (in the same basis).  Given $A_i$ are the matrix representing the morphism  for  $K$-homology  calculated in \cite{Sanchez-Garcia(2006SL)}, one  has  that  $A_i^T$ are the matrices representing the morphism $\Phi_i$. 
 
We calculate the elementary divisors of the matrices representing the morphism $\Phi_i$. Using that we obtain

$$H^p(\underbar{E}\sldrei,\calr)= 0\text{, if } p>0 \text{ , and}$$
$$H^0(\underbar{E}\sldrei,\calr)\cong\IZ^{\oplus8}.$$

As the Bredon cohomology concentrates at low degree, using the same argument that in \cite{Sanchez-Garcia(2006SL)} we conclude that

$$K^0_{\sldrei}(\underbar{E}\sldrei)\cong\IZ^{\oplus8},$$

$$K^1_{\sldrei}(\underbar{E}\sldrei)= 0.$$
This  agrees  with  the   result predicted  by  Theorem \ref{theoremUCT}

\section{Twisted Equivariant  $K$-Theory.}

In this section we use the spectral described in Theorem \ref{spectral} to calculate the equivariant twisted K-theory groups ${}^{u_1}K_{\sldrei}(\underbar{E}\sldrei)$.

The cochain complex associated to the $\sldrei$-CW-complex structure of $\underbar{E}\sldrei$ described in section \ref{sectiontwists} is:
\begin{multline*}
0\rightarrow\bigoplus_{k=1}^5R_{u_1}(stab(v_k))\\ \xrightarrow{\Phi_1}\bigoplus_{j=1}^8R_{u_1}(stab(e_j))\xrightarrow{\Phi_2}
\bigoplus_{i=1}^5R_{u_1}(stab(t_i)) \\ \xrightarrow{\Phi_3} R_{u_1}(stab(T))\rightarrow 0
\end{multline*}

As the class $u_1$ restricts to $0$ at the level of $2$-cells and $3$-cells,  we have the following result
\begin{proposition}For $*=3$, we have a natural isomorphism
 $$H^*(\underbar{E}\sldrei ;\calr_{u_1})\cong H^*(\underbar{E}\sldrei ;\calr).$$
\end{proposition}
We only have to determine $\Phi_1$ and $\Phi_2$, because $\Phi_3$ was determined in section \ref{sectionuntwisted}.

\subsection{Determination of $\Phi_1$}
We know from section \ref{sectiontwists} that the class $u_1$ comes from the the subgroup $stab(v_1)=\langle g_2,g_3\rangle\cong S_4$.  We have an inclusion 
$$stab(e_1)\xrightarrow{i} stab(v_1)$$$$\langle g_2,g_5\rangle\rightarrow\langle g_2,g_3\rangle$$$$g_2\mapsto g_2$$$$g_5\mapsto g_3g_2g_3^{-1}g_2g_3,$$ 
consider the central extension associated to a representing cocycle the unique non-trivial 3-cohomology class in $S_4$ (say $u_{1\mid}$) and its restriction to 
$C_2\times C_2$

$$\begin{diagram}
\node{0}\arrow{e}\node{\IZ_2}\arrow{e}\node{S_4^*}\arrow{e}\node{S_4}\arrow{e}\node{0}\\
\node{0}\arrow{e}\node{\IZ_2}\arrow{e}\arrow{n}\node{G^*}\arrow{n,r}{j}\arrow{e}\node{C_2\times C_2}\arrow{n,r}{i}\arrow{e}\node{0}
\end{diagram}$$

To evaluate the morphism $i^*:R_{u_{1\mid}}(S_4)\rightarrow R_{u_{1\mid}}(C_2\times C_2)$ is equivalent to determine the morphism
$j^*:R(S_4^*)\rightarrow R(G^*)$. 

We have a description of $S_4^*$ in terms of generators and relations given in section \ref{sectionrepresentation}.  In terms of these generators the group $G^*$ can be described
as follows
$$D_2^*=\langle h_1, h_3, z\rangle$$
and the homomorphism $j:D_2^*\rightarrow S_4^*$ is the inclusion. 

 Using the character tables  from section \ref{sectionrepresentation},  the following fact is straightforward 
 $$j^*(\epsilon_1)=\eta_1,$$
 $$j^*(\epsilon_2)=\eta_4,$$
 $$j^*(\epsilon_3)=\eta_1+\eta_4,$$
 $$j^*(\epsilon_4)=\eta_1+\eta_2+\eta_3,$$
 $$j^*(\epsilon_5)=\eta_2+\eta_3+\eta_4,$$
 $$j^*(\epsilon_6)=\eta_5,$$
 $$j^*(\epsilon_7)=\eta_5,$$
 $$j^*(\epsilon_8)=2\eta_5.$$
 
 The matrix corresponding to the morphism $i^*:R_{u_{1\mid}}(stab(v_1))\rightarrow R_{u_{1\mid}}(stab(e_1))$ is given by
 $$\begin{pmatrix}
 1&0&0&0&0\\
 0&0&0&1&0\\
 1&0&0&1&0\\
 1&1&1&0&0\\
 0&1&1&1&0\\
 0&0&0&0&1\\
 0&0&0&0&1\\
 0&0&0&0&2\\
 \end{pmatrix}$$
 
As the class $u_1$ restricts non-trivially to $stab(e_1)$, we have to determine the morphism 
 $i^*:R_{u_{1\mid}}(stab(v_2))\rightarrow R_{u_{1\mid}}(stab(e_1))$.
  
Using the above character table  one  has: 

$$i^*(\gamma_1)=\eta_1,$$
 $$i^*(\gamma_2)=\eta_4,$$
 $$i^*(\gamma_3)=\eta_3,$$
 $$i^*(\gamma_4)=\eta_2,$$
 $$i^*(\gamma_5)=\eta_2+\eta_3,$$
 $$i^*(\gamma_6)=\eta_1+\eta_4,$$
 $$i^*(\gamma_7)=\eta_5,$$
 $$i^*(\gamma_8)=\eta_5$$
 $$i^*(\gamma_9)=\eta_5.$$

 The  matrix corresponding to the morphism $i^*:R_{u_{1\mid}}(stab(v_2))\rightarrow R_{u_{1\mid}}(stab(e_1))$ is given by
 $$\begin{pmatrix}
 1&0&0&0&0\\
 0&0&0&1&0\\
 0&0&1&0&0\\
 0&1&0&0&0\\
 0&1&1&0&0\\
 1&0&0&1&0\\
 0&0&0&0&1\\
 0&0&0&0&1\\
 0&0&0&0&1\\
 \end{pmatrix}$$
 
 Now, we have to determine the morphisms $$R_{u_1\mid}(stab(v_1))\rightarrow R_{u_1\mid}(stab(e_i))$$
 with $i=2,3$, note that $R_{u_1\mid}(stab(e_i))\cong R(stab(e_i))$ because $H^3(D_3;\IZ)$ is trivial.
 
 The inclusion $stab(e_2)\rightarrow stab(v_1)$ is given by
 
 $$\langle g_6,g_{10}\rangle\rightarrow\langle g_2,g_3\rangle$$
 $$g_6\mapsto g_3 g_2g_3^{-1}$$
 $$g_{10}\mapsto g_2g_3g_2g_3^{-1}.$$
This map induces a map
$i:stab(e_2)^*\rightarrow stab(v_1)^*$, where $G^*$ denotes the inverse image of $G\subseteq S_4$ by the covering map 
$S_4^*\rightarrow S_4$. Is easy to see that  $i(g_6)\sim g_2$, $i(g_{10})\sim zg_2$ and $i(g_6g_{10})\sim zg_2g_3$. Note that $S_4$ 
has three non isomorphic irreducible projective characters,  namely $\epsilon_6$, $\epsilon_7$ and $\epsilon_8$. We have that
$$i^*(\epsilon_6)(1)=2, i^*(\epsilon_6)(g_6)=0, \text{ and } i^*(\epsilon_6)(g_6g_{10})=-1,$$
and then $i^*(\epsilon_6)=\lambda_3$. In the same way we obtain 
$i^*(\epsilon_7)=\lambda_3$. On the other hand 
$$i^*(\epsilon_8)(1)=4, i^*(\epsilon_8)(g_6)=0, \text{ and } i^*(\epsilon_8)(g_6g_{10})=1,$$ and then 
$i^*(\epsilon_8)=\lambda_1+\lambda_2+\lambda_3.$
Combining the above results with the obtained in \cite{Sanchez-Garcia(2006SL)} for linear characters we obtain that the matrix 
corresponding to the morphism $$R_{u_1\mid}(stab(v_1))\rightarrow R(stab(e_2))$$ is given by
$$\begin{pmatrix}
 1&0&0\\
 0&1&0\\
 0&0&1\\
 1&0&1\\
 0&1&1\\
 0&0&1\\
 0&0&1\\
 1&1&1\\
 \end{pmatrix}$$

Now we  determine the morphism $$i^*R_{u_1\mid}(stab(v_1))\rightarrow R(stab(e_3))$$

The inclusion $stab(e_3)\rightarrow stab(v_1)$ is given by
$$\langle g_6,g_5\rangle\rightarrow\langle g_2,g_3\rangle$$
 $$g_6\mapsto g_3 g_2g_3^{-1}$$
 $$g_5\mapsto g_3g_2g_3^{-1}g_2g_3.$$
this map induces a map
$i:stab(e_3)^*\rightarrow stab(v_1)^*$. Is easy to see that  $i(g_6)\sim g_2$, and $i(g_6g_5)\sim zg_2g_3$, we have that
$$i^*(\epsilon_6)(1)=2, i^*(\epsilon_6)(g_6)=0, \text{ and } i^*(\epsilon_6)(g_6g_5)=-1,$$ and then $i^*(\epsilon_6)=\lambda_3$. 
 In the same way we obtain 
$i^*(\epsilon_7)=\lambda_3$. On the other hand 
$$i^*(\epsilon_8)(1)=4, i^*(\epsilon_8)(g_6)=0, \text{ and } i^*(\epsilon_8)(g_6g_{10})=1,$$ and then 
$i^*(\epsilon_8)=\lambda_1+\lambda_2+\lambda_3.$
Combining the above results with the obtained in \cite{Sanchez-Garcia(2006SL)} for linear characters we obtain that the matrix 
corresponding to the morphism $$R_{u_1\mid}(stab(v_1))\rightarrow R(stab(e_3))$$ is given by
$$\begin{pmatrix}
1&0&0\\
 0&1&0\\
 0&0&1\\
 1&0&1\\
 0&1&1\\
 0&0&1\\
 0&0&1\\
 1&1&1\\
 \end{pmatrix}$$
We have to determine now the morphism $$R_{u_1\mid}(stab(v_2))\rightarrow R(stab(e_i)),$$
for $i=4,5$, note that $R_{u_1\mid}(stab(e_i))\cong R(stab(e_i))$ because $H^3(C_2;\IZ)$ is trivial.

 The inclusion $stab(e_4)\rightarrow stab(v_2)$ is given by
 $$\langle g_6\rangle\rightarrow\langle g_4,g_5\rangle$$
 $$g_6\mapsto g_5( g_4g_5)^3$$
 this map induces a map
$i:stab(e_4)^*\rightarrow stab(v_2)^*$. It is clear that  $i(g_6)\sim ab$ , in the notation used for character table of $D_6^*$, 

$$i^*(\gamma_7)(1)=2\text{ and }i^*(\gamma_7)(g_6)=0,$$ and then $i^*(\gamma_7)=\rho_1+\rho_2$, in the same way we obtain 
$i^*(\gamma_8)=\rho_1+\rho_2$ and $i^*(\gamma_9)=\rho_1+\rho_2$. Combining the above results with the obtained in 
\cite{Sanchez-Garcia(2006SL)} for linear characters we obtain that the matrix 
corresponding to the morphism $$R_{u_1\mid}(stab(v_2))\rightarrow R(stab(e_4))$$ is given by

$$\begin{pmatrix}
1&0\\
 0&1\\
 0&1\\
 1&0\\
 1&1\\
 1&1\\
 1&1\\
 1&1\\
 1&1\\
 \end{pmatrix}$$
 
 The inclusion $stab(e_5)\rightarrow stab(v_2)$ is given by
 $$\langle g_5\rangle\rightarrow\langle g_4,g_5\rangle$$
 $$g_5\mapsto g_5$$
 this map induces a map
$i:stab(e_5)^*\rightarrow stab(v_2)^*$. Again  in  the  notation  used  for  the  character  table  of  $D_6^*$,  $i(g_5)\sim b$.
We have that
$$i^*(\gamma_7)(1)=2\text{ and }i^*(\gamma_7)(g_5)=0,$$ and then $i^*(\gamma_7)=\rho_1+\rho_2$, in the same way we obtain 
$i^*(\gamma_8)=\rho_1+\rho_2$ and $i^*(\gamma_9)=\rho_1+\rho_2$. Combining the above results with the obtained in 
\cite{Sanchez-Garcia(2006SL)} for linear characters we obtain that the matrix 
corresponding to the morphism $$R_{u_1\mid}(stab(v_2))\rightarrow R(stab(e_4))$$ is given by

$$\begin{pmatrix}
1&0\\
 0&1\\
 1&0\\
 0&1\\
 1&1\\
 1&1\\
 1&1\\
 1&1\\
 1&1\\
 \end{pmatrix}$$
 
The elementary divisors of the matrix representing 
the morphism $\phi$ are 1, 1, 1, 1, 1, 1, 1, 1, 1, 1, 1, 1, 1, 1, 1, 1, 1, 1, 1. The rank of this matrix is 19.

\subsection{Determination of $\Phi_2$}

In order to determine $\Phi_2$ notice that the unique morphisms that differ from the untwisted case are 
$i^*:R_{u_1\mid}(stab(e_1))\rightarrow R_{u_1\mid}(stab(t_1))$ and $i^*:R_{u_1\mid}(stab(e_1))\rightarrow R_{u_1\mid}(stab(t_4))$
in both cases we have that $i^*(\eta_5)=\rho_1+\rho_2$.  We obtain that the elementary divisors of the matrix representing 
the morphism $\phi$ are 1, 1, 1, 1, 1, 1, 1, 1. The rank of this matrix is 10.

As  the  cochain  complexes  involved  are  free, the  computation  of  ranks  and  elementary  divisors  yield 
\begin{theorem}\label{equationktheory}
The  Bredon cohomology  with  coefficients in twisted  representations  satisfies
$$H^p(\underbar{E}\sldrei,\calr_{u_1})= 0\text{\, if } p>0,\quad  
H^0(\underbar{E}\sldrei,\calr_{u_1})\cong\IZ^{\oplus13}.$$
\end{theorem}

Since the Bredon cohomology concentrates at low degree, the spectral sequence described in section \ref{spectral} collapses at level 2 and  we conclude
\begin{theorem}\label{theoremtwistedktheory}
$${ }^{u_1}K^0_{\sldrei}(\underbar{E}\sldrei)\cong\IZ^{\oplus13},$$

$${ }^{u_1}K^1_{\sldrei}(\underbar{E}\sldrei)= 0.$$
\end{theorem}

\section{Twisted  equivariant  $K$-Homology and  relation  to the  Baum-Connes  Conjecture  with  coefficients}

The  Baum-Connes  Conjecture  \cite{baumconnes}, \cite{valettemislin}  predicts  for  a   discrete   group $G$ the  existence  of  an  isomorphism 
$$\mu_{i}: K^G_i(\eub{G})\to K_i(C_{r}^*(G))$$
 given  by  the  (analytical) assembly  map,  where  $C_r^ *(G)$  is  the   reduced  $C^*$-algebra of  the  group  $G$.
 
More  generally,  given any  $G$-$C*$-Algebra,  the  Baum-Connes   conjecture  with  coefficients   predicts  that  a  map  
$$\mu_{i}: K^G_i(\eub{G}, A)\to K_i( A\rtimes G ))$$
Where $ K^G_i(\eub{G}, A)$ is  defined  in  terms  of  equivariant  and  bivariant  $KK$-groups ,

$$K_*^G(\eub{G},A)= \underset{{\rm G-compact}X\subset \eub{G}}{\colim} KK_{*}(C_{0}(X), A)$$

 and  $  A\rtimes G$ denotes  the  crossed  product $C^ *$-algebra. 
See \cite{echterhoffchabert},  \cite{echterhoff}  for  more  details. 

The  class  of  twists considered  in the  example  of  $\sldrei$  let  define  a  particular  choice  of  coefficients  for  this  assembly  map.

\begin{definition}
Let $G$  be  a  discrete  group. Given a  cocycle  $\omega \in Z^2(G,S^1)$,  an  $\omega$-representation  on a  Hilbert  space  $\HH$  is  a   map  $V: G\to  \UU(\HH)$ 
satisfying $v(s)V(t)= \omega(s,t)V(st)$.
\end{definition}
 
Composing  with the  quotient  map $\UU(\HH)\to  PU(\HH)= \UU(\HH)/S^1$,  together  with the  identification of  the group $PU(\HH)$ as  the  automorphisms  of  the  $C^*$-algebra  of  compact  operators on $\HH$, denoted  by $\KK$ gives  an  action  of  $G$ on  $\KK$.     This  algebra   is  denoted as  $\KK_\omega$. Theorem \ref{theoremUCT}  together  with the Atiyah-Hirzebruch spectral sequence  yield:

\begin{corollary}

Let  $G$  be  a  discrete  group  with a finite model  for  $\eub{G}$. Let  $\omega \in Z^2(G, S^1)$.  Assume  that the  Bredon  homology  groups  $H_{*}^G(X, R^{- \alpha})$ are  all  free  abelian and are concentrated in degree 0 and 1.  Then  there  exists  a  duality  isomorphism
$$K^{*}_G(C_0(\eub{G}), K_{-\omega}) \longrightarrow K^{*}_G( K_{\omega}, C_0(\eub{G}))$$

\end{corollary}

Related  duality  isomorphisms  have  been deduced  for  almost  connected  groups  by  Echterhoff  \cite{echterhoff} using  methods  from  equivariant Kasparov  $KK$-theory (particularly the Dirac-Dual Dirac  method) and  positive  results  for  the  Baum-Connes  Conjecture  with coefficients \cite{echterhoffemersonkim}.

As  a  consequence  of  Theorem  \ref{theoremUCT}, and  Theorem \ref{equationktheory},   the  $\alpha$- twisted  Bredon homology  of $\eub{\sldrei}$ is  given by   
\begin{equation*}
H_p(\underbar{E}\sldrei,\mathcal{R}_{u_1})= 0\text{, if }\, p>0,\quad  
H_0(\underbar{E}\sldrei,\mathcal{R}_{u_1})\cong\IZ^{\oplus13}.
\end{equation*}

A  spectral  sequence  argument and  the  particular  shape  of  the  twists   let  us  conclude  that  these  groups  agree  with  the equivariant $K$-homology  groups  with  coefficients in $\KK_{u_1}$, appearing  in the left  hand  side  of  the conjecture  with  coefficients. 

\begin{corollary}\label{corollaryduality}
The  equivariant $K$-homology  groups of  $\sldrei$ with  coefficients  in the  $\sldrei $-$C^*$  algebra  $\KK_{u_{1}}$  are  given  as  follows: 
\begin{equation*}
K_p^{\sldrei}(\eub{\sldrei},\KK_{u_{1}})= 0 \, \text{ p odd}, \quad K_p^ {\sldrei}(\eub{\sldrei}, \KK_{u_{1}})\cong\IZ^{\oplus13} \text{ p even}
\end{equation*}

\end{corollary}
\newpage

\bibliographystyle{alpha}
\bibliography{twisted}

\end{document}